\documentclass[11pt]{article}
\usepackage{amssymb}
\usepackage{amsmath}

\usepackage{algorithmic}
\usepackage{algorithm}

\usepackage{amssymb}
\usepackage{amsmath}

\newenvironment{proof}{\noindent {\bf Proof }}
{\hfill $\bullet$ \vspace{0.25cm}}

\def\one{{\bf 1}\hskip-.5mm}

\def\R{{\mathbb R}}
\def\Z{{\mathbb Z}}

\def\ZZ{{\mathbb Z}}
\def\N{{\mathbb N}}
\def\LL{{\mathcal G}}
\def\SS{{ S}}
 \def\F {{\mathcal F}}
\def\AA{{\mathcal A}}

\newtheorem{theo}{Theorem}
\newtheorem{prop}{\indent Proposition}

\newtheorem{rem}{\indent Remark}
\newtheorem{lem}{\indent Lemma}
\newtheorem{defin}{\indent Definition}
\newtheorem{cor}{\indent Corollary}

\title{ Perfect simulation of  infinite
  range Gibbs measures and coupling with  their finite range 
   approximations }

\author{A.~Galves \and E.~L\"ocherbach \and E.~Orlandi}

\date{September 28, 2009}

\begin{document}

\maketitle
\begin{abstract}

In this paper we address  the   questions    of perfectly sampling   a
Gibbs measure with infinite range interactions  and  of  perfectly sampling the measure  together with its finite
range approximations. We solve these questions by introducing a perfect
simulation algorithm for the  measure and for the  coupled measures.   The algorithm works for general Gibbsian interaction under   requirements on  the tails of the interaction.   As a consequence we
obtain 
an upper bound for the error we make when sampling from a
finite range approximation instead of the true infinite range measure.
 \end{abstract}



\section{Introduction}
Consider a Gibbs measure with an infinite range potential and
its finite range approximation obtained by truncating the 
interaction at a certain range.  Can we construct a perfect sampling of the Gibbs measure with an infinite range potential?
Further, if we make a local inspection of a
perfect sampling of the finite range approximation, how often does it
coincide with a sample from the original infinite range measure?  We
address these questions by introducing a new   perfect simulation
algorithm for the original  Gibbs measure   and a 
coupled  perfect simulation
algorithm for the original measure together with its finite range approximation.

By a  perfect simulation algorithm we mean a function mapping a
family of independent uniform random variables on the set of 
configurations having as  law the   infinite range Gibbs 
measure.  
By a coupled perfect simulation algorithm we mean a function mapping a
family of independent uniform random variables on the set of pairs of
configurations having as first marginal the original infinite range
measure and as second marginal the finite range approximation. These
functions  are defined through two coupled Glauber dynamics having as reversible
measures the original Gibbs measure and its finite range approximation,
respectively. The perfect simulation algorithm produces coupled samples
of the two measures. 
      
The main difficulty in this constructive approach is that we have to
deal with infinite range spin flip rates. This difficulty is overcome
by a decomposition of the spin flip rates of the Glauber dynamics
as a convex sum of local
spin flip rates. 

Our results are obtained under high temperature conditions and
assuming that the interaction decays fast enough.   We do not need any monotonicity  property of the infinite range Gibbs measure.
  Under  these
conditions, we show that the algorithm that we propose stops almost
surely, after a finite number of steps.  Moreover,  we give an upper bound, uniform in space, for the
probability of local discrepancy between the coupled realizations of
the two measures. This is the content of Theorem \ref{theo:perfect}.
As a corollary, Theorem \ref{theo:coupling} shows that the same upper
bound holds for Ornstein's $\bar d$-distance between the two measures.    

An upper bound for the $ \bar d $-distance   slightly worse  than the one obtained in 
Theorem \ref{theo:coupling} can be
obtained under less restrictive assumptions, by  applying  Dobrushin's
contraction method.  We refer the reader to the series of papers of
Dobrushin (1968-1970) as well as to Presutti (2009) for a recent nice
self-contained presentation of the subject. However, the contraction
method is not constructive and does not provide an explicit sampling
procedure for the measures.  
 
The approach followed in the present paper was suggested by a recent work of Galves et al. (2008). There, the flip rates of an infinite range interacting multicolor system were decomposed as convex combinations of local range flip rates. With respect to this work, the novelty of what follows is that  the decomposition is explicit and more suitable in view of the intended coupling result. The idea of decomposing flip rates can be traced back to Ferrari (1990)  in which something with the same flavor was done for a finite range interacting particle system and to Ferrari et al. (2000) in which the probability transition of a stochastic chain of infinite order was decomposed as a convex combination of finite order probability transitions.  We refer to Galves et
al. (2008) for more details and related literature.

This paper is organized as follows. In Section 2, we present the basic
notations and the main result, Theorem \ref{theo:perfect}.  In Section
3, we introduce the Glauber dynamics and the associated processes.
The representation of the spin flip rates as convex combination of
local rates is given in Theorem \ref{theo:dec}.  In  Section
4.1  we introduce the backward sketch process  which is the basis of the 
perfect simulation algorithm 
and collect  the main properties of this process
in Proposition \ref{nstop}. Section 4.2 presents  
the coupled perfect simulation algorithm for the couple of
measures.  In Section 5 we prove Theorem \ref{theo:perfect}.  We
conclude the paper in Section 6  discussing how certain features  of  high temperature Gibbs systems  can be obtained applying our construction.    In particular we  give results on  
Ornstein's $\bar d$-distance  between  the infinite range and the finite range Gibbs measure,  and on the decay of correlations.   
These results are obtained under  requirements stronger than  the one  needed using the   Dobrushin contraction method,  since   they are based on our     explicit coupled perfect
simulation procedure.

\vskip0.5cm \noindent {\bf Acknowledgments.}  The authors thank
L. Bertini, P. Collet, P. Ferrari, R. Fern\'andez, D. Gabrielli  and E. Presutti for helpful
discussions. We thank the Institut Henri Poincar\'e - Centre
Emile Borel (workshop 
     M\'ecanique Statistique, Probabilit\'es et Syst\`emes de Particules 2008)
 for hospitality and support.
 
The authors have been partially supported by Prin07: 20078XYHYS (E.O.),
by CNPq 305447/2008-4 (A.G.),  ANR-08-BLAN-0220-01 (E.L.) and
Galileo 2008-09 (E.L. and  E.O.).  

This paper is part of PRONEX/FAPESP's project \emph{Stochastic
behavior, critical phenomena and rhythmic pattern identification in
natural languages} (grant number 03/09930-9) and CNRS-FAPESP project
\emph{Probabilistic phonology of rhythm}. 
Ce travail a b\'en\'efici\'e d'une aide de l'Agence Nationale de la Recherche
portant la r\'ef\'erence ANR-08-BLAN-0220-01. 
 
\section{Definitions and main results}
Let $A= \{-1,1\}$ and $S=A^{\Z^d} $ be the set of spin
configurations.  We endow $S$ with the product sigma algebra.  We
define on $\Z^d$ the $L^1$ norm, $\| i \|= \sum_{k=1}^d |i_k|$. 

Configurations will be denoted by Greek letters $\sigma, \eta, \dots
$. A point $ i \in \ZZ^d $ will be called a site.  For any $i \in
\ZZ^d$, notations $\sigma(i)$ and $\sigma_i$ will be used
in order to denote the value of the configuration $\sigma$ at
site $i$. By extension, for any subset $V \subset \ZZ^d$,
$\sigma(V) $ and $\sigma_V \in A^V$ denote the restriction of the configuration
$\sigma$ to the set of positions in $V.$ If $V \subset \ZZ^d$ is a finite
subset, we denote $V \Subset \ZZ^d .$ In this case, we write $| V| $ for the 
cardinal of $V.$ Finally, for any $ i \in \ZZ^d $ and $\sigma \in S,$ we shall
denote $\sigma^i $ the modified configuration
$$ \sigma^i (j) = \sigma (j) , \mbox{ for all } j \neq i, \mbox{ and } \sigma^i (i) = - \sigma (i) .$$ 

\begin{defin}
  An interaction is a collection $ \{J_B  , B  \Subset \ZZ^d \} $ of real  numbers 
  indexed by finite subsets $ B \Subset \ZZ^d $, $ |B| \ge 2$,  which satisfies 
 \begin{equation}\label{B3}
 \sup_{i \in \Z^d}  \sum_{ B : i \in B } |J_B|  < \infty .
\end{equation} 
For finite $B\subset \Z^d$ and $ \sigma \in S$ let
$$ \chi_B(\sigma) = \prod_{i\in B} \sigma (i). $$
\end{defin}

\begin{defin}
A probability measure $\mu$ on $S$ is said to be a Gibbs measure
  relative to the interaction $\{J_B, B \Subset  \ZZ^d \}$ if for all $i \in \ZZ^d$ and
  for any fixed $\zeta \in S,$ a version of the conditional
  probability $ \mu (\{ \sigma : \sigma (i)= \zeta (i) | \sigma (j) =
  \zeta (j) \mbox { for all } \ j \neq i \})$ is given by
\begin{equation}
  \label{B4}  
\mu (\{ \sigma : \sigma (i)= \zeta (i) | \sigma (j) = \zeta (j) \mbox
{ for all } \ j \neq i \} ) = \frac{1}{1 + \exp ( - 2 \sum_{ B: i \in B } J_B \chi_B(\zeta)  ) } . 
\end{equation}
\end{defin}
To  define  the Gibbs measure $\mu^{[L]} $ relative to the 
interaction truncated at range $L $ we set 
\begin{equation}
  \label{RR1}   B_i (L) = \{ j \in \ZZ^d :   \| j - i \| \le L \} .\end{equation}
 Then     $\mu^{[L]}$ is the probability measure
on $S$ such that for all $i \in \ZZ^d$ and for any fixed
$\zeta \in S,$ a version of the conditional probability $ \mu^{[L]}
(\{ \sigma : \sigma (i)= \zeta (i) | \sigma (j) = \zeta (j) \mbox {
for all } \ j \neq i \})$ is given by
\begin{multline*}
\mu^{[L]} (\{ \sigma : \sigma (i)= \zeta (i) | \sigma (j) = \zeta (j)
\mbox { for all } \ j \neq i \} ) \\
= \frac{1}{1 + \exp ( -2 \sum_{ B : i \in B, B \subset B_i (L) } J_B \chi_B(\zeta) )} .  
\end{multline*}
We consider the interaction $J_B^\beta  = \beta \,J_B  $ where $
\beta$ is a positive parameter.
The associated Gibbs measures will be denoted by $\mu$ and $\mu^{[L]} $,
 omitting to write the explicit dependence on $\beta$.  We denote by $
 (X(i), i \in \ZZ ^d ), $ $ (X^{[ L]}(i), i \in \ZZ ^d ) $ two random
 configurations on $S$ distributed according to $\mu $ and
 $\mu^{[L]},$ respectively.

All processes that we consider in this paper will be constructed as
functions of an independent family of uniform random variables 
$(U_n (i), n \in \N , i \in \Z^d ) $. Denote $(\Omega, \AA, P)$ 
the probability space on which is
defined the independent family of uniform random variables.  

\begin{defin}
A perfect sampling algorithm of the measure $\mu $ 
is a map $F $ from $ [0, 1 ]^{ \N \times \Z^d
  } $ to $ \{ -1 , + 1 \}^{\Z^d },$ such that
$$ F (U_n (j), n\in \N, j \in \Z^d) \mbox{ has distribution  } \mu .$$
A coupled perfect sampling algorithm for the pair $((X(i), X^{[L]}
  (i)), i \in \ZZ^d ) $ is a map $G $ from $ [0, 1 ]^{ \N \times \Z^d
  } $ to $ \{ -1 , + 1 \}^{\Z^d \times \Z^d},$ such that
$$ G (U_n (j), n\in \N, j \in \Z^d) \mbox{ has marginals  } (X(i), i
  \in \ZZ^d ) \mbox{ and } (  X^{[L]} (i), i \in \ZZ^d )   .$$
\end{defin}
Since $F$ produces     an unbiased sample from the target distribution $\mu$, the algorithm is denoted  ``perfect''. 
\begin{defin} \label {S.2} 
The  perfect sampling algorithm  
  $F $ of  the measure $\mu $  is said to stop almost surely after a finite number of steps if 
 for every site  $i \in \Z^d$, there exists a   finite subset
  $  F^{(i)}  $ of $\Z^d$ and a finite random variable  $ N^{(i)}_{STOP} \ge 1$ such that if
\begin{equation}\label{RR2}
 U'_n   (j) = U_{n} (j) \mbox{ for all } j \in  F^{(i)}, 1 \le n \le
N^{(i)}_{STOP} ,
\end{equation}
then
\begin{equation}\label{RR3}
 F (U_n (j), n\in \N, j \in \Z^d) (i) =  F (U'_n (j), n \in \N, j \in \Z^d) (i).
\end{equation}
\end{defin}
Note that (\ref{RR3})   implies that $F (U_n (j), n\in \N, j \in \Z^d) (i) $  depends only on    $U_n (j)$ for $ 1 \le n \le
N^{(i)}_{STOP}$. 
The following theorem is our main result.
\begin{theo}\label{theo:perfect}
Assume that
\begin{equation}
  \label{cn1}  
\sup_{i \in \Z^d} \sum_k | B_i (k) | \left(  \sum_{ B : i \in B , B \subset B_i (k) , B \not \subset   B_i (k- 1) } | J_B|  \right)< \infty .
\end{equation} 
\begin{enumerate}
\item  There exists $\beta_c > 0$ such that
   for any $\beta < \beta_c$
   there is a perfect sampling algorithm of $\mu.$ 
   This algorithm stops after a finite number of steps almost surely and 
 \begin{equation}
  \label{s1}    \sup_{i \in \Z^d} E( N^{(i)}_{STOP}) \le \frac 1 \gamma,
\end{equation} 
where $\gamma $ is given  in \eqref{eq:epsilon}. 
\item  Moreover, for any $\beta < \beta_c,$ there exists
   a
   coupled perfect sampling algorithm for the pair
   $ ((X(i), X^{[L]} (i)), i \in \ZZ^d ) $ satisfying  
\begin{equation}
  \label{cn2}   
\sup_{i \in \Z^d } P( X(i) \neq X^{[L]}(i) ) \le \frac 1 \gamma \sup_{i \in \Z^d }
\left( 1- e^{-\beta \sum_{ B : i \in B , B \not\subset B_i (L) } | J_B|  } \right).
\end{equation} 
 \end {enumerate}
\end{theo}

\begin{rem}
The constant $\beta_c$ is the solution 
of  equation \eqref {betac1} in Section \ref{sectionafter}. \end{rem}

\begin{rem}
In case of an interaction defined through a pairwise potential
$\{ J(i,j),   (i,j) \in \ZZ^d \times \ZZ^d\} $ with $ J(i,i)=0$ for $i \in \ZZ^d$, 
condition \eqref{cn1} of
Theorem \ref{theo:perfect} reads as 
$$ \sup_{i \in \Z^d} \sum_k | B_i (k) | \left(  \sum_{ j : \| i - j\| = k  } | J(i,j)|   \right)< \infty .$$
In this case,
$$ \sup_{i \in \Z^d } P( X(i) \neq X^{[L]}(i) ) \le \frac 1 \gamma \sup_{i \in \Z^d }
\left( 1- e^{-\beta \sum_{ j : \| i - j \| > L}  | J(i,j)|  } \right) .$$
\end{rem}

The above   perfect sampling algorithms  are  based on a coupled
construction of two processes having $\mu $ and $\mu^{[L]} $ as
reversible measures.  These two processes will be introduced in the
next section. The definition of the algorithms can only be given in
Section \ref{algo}.  The
proof of Theorem
\ref{theo:perfect} is given in Section \ref{sectionafter}.

\section{Glauber dynamics}
We now introduce a Glauber dynamics having $\mu $ as reversible
measure.  This is an interacting particle system $(\sigma_t (i), i \in
\Z^d, t \in \R )$ taking values in $S.$ Sometimes we will also use the
short notation $(\sigma_t)_t$ for the interacting particle system.  To
describe the process, we need some extra notation.

For any $i \in \ZZ^d,$  $\sigma \in \{-1,1\}^{\ZZ^d},$  and $\beta>0, $ we
define $ c_i (\sigma),$ which is the rate at which the spin $i$ flips when the
  system is in the configuration $\sigma ,$ as 
 \begin{equation}
  \label{A2}c_i  (\sigma) =  \exp \left( -\beta   \sum_{ B : i \in B} J_B  \chi_B (\sigma)  \right) .
  \end{equation} 
 The generator $\LL$ of the
  process $(\sigma_t)_t$ is defined on cylinder functions $f : S \to
  \R $ as follows
  \begin{equation}
  \label{eq:generator}
  \LL  \,f(\sigma) \,=\, \sum_{i \in \ZZ^d}    c_i (\sigma) [f(\sigma^{i}) - f(\sigma)]  \, .
\end{equation}  
Observe that the condition \eqref{cn1} of Theorem \ref{theo:perfect} implies that
$$ \sup_i \sum_{ B : i \in B} |B| \, | J_B| < \infty,$$
where $|B| $ designs the cardinal of the set $B.$ Hence the rates $c_i
(\sigma)$ are uniformly bounded in $i$ and $\sigma ,$ and
 \begin{equation}
   \label{A3} 
   \sup_{i \in \ZZ^d} \sum_{j \in \Z^d} \sup_{\sigma } | c_i  (\sigma)- c_i  (\sigma^j)| < \infty. 
\end{equation} 
Therefore, Theorem 3.9 of Chapter 1 of Liggett (1985) implies that
$\LL$ is the generator of a Markov process $(\sigma_t )_t$ on $\SS .$
By construction, the process $(\sigma_t)_t$ is reversible
with respect to the Gibbs measure $\mu $ corresponding to the interaction
$ J_B^\beta = \beta \, J_B  .$

For any $ \ell \ge 1$, we now define a Glauber dynamics having
$\mu^{[\ell]}$ as reversible measure. More precisely, we consider the
Markov process $(\sigma^{[\ell]}_t)_t $ on $\SS $ having generator
\begin{equation}
  \label{A8}
    \LL^{[\ell]} \,f(\sigma) \,=\, \sum_{i \in \ZZ^d}    c_i^{[\ell]} (\sigma) [f(\sigma^{i}) - f(\sigma)], 
  \end{equation}
where the rates $ c_i^{[\ell]} $ are given by
  \begin{equation}
  \label{A9}   c_i^{[\ell]} (\sigma) =  r_i^{[\ell]} \exp \left( - \beta  \sum_{ B : i \in B , B \subset B_i (\ell ) }
  J_B \chi_B (\sigma)  \right) 
\end{equation}
and
$$ r_i^{[\ell]} = e^{ - \beta \sum_{ B : i \in B , B \not \subset B_i( \ell)  }| J_B| } .$$ Let us stress the fact that compared to the usual
definition of a Glauber dynamics, an extra factor $ r_i^{[\ell]} $
appears in the definition of the rates $ c_i^{[\ell]} .$ This extra
factor is important in view of the intended coupling of both processes
$\sigma_t$ and $\sigma_t^{[\ell]} $ (see Theorem \ref{theo:dec} below)
and does not change the equilibrium behavior of the process as is
shown in the next proposition.
 
\begin{prop}
  The process $(\sigma^{[\ell]}_t)_t $ is reversible with respect to
  the Gibbs measure $\mu^{[\ell]}$ relative to the pairwise interaction
  $J_\beta $ truncated at range $\ell .$
\end{prop}

\begin{proof}
By construction, we have that
$$ \frac{ c_i^{[\ell]} (\sigma)}{ c_i^{[\ell]} (\sigma^i)} 
= \frac{\mu^{[\ell]} (\sigma^ i)}{\mu^{[\ell]} (\sigma)}.$$

Thus, by Proposition 2.7 of Chapter IV of Liggett (1985), it follows
that $(\sigma^{[\ell]}_t)_t $ is reversible with respect to
$\mu^{[\ell]}. $
\end{proof}

The main tool to prove Theorem \ref{theo:perfect} is a coupled
 construction of the processes $(\sigma_t (i), i \in \Z^d, t \in \R )$
 and $(\sigma^{[\ell]}_t (i), i \in \Z^d, t \in \R )$ for $\ell=L,$
 using the basic family of uniform random variables $(U_n (i), n \in
 \N, i \in \Z^d ) .$ This coupling is based on a decomposition of the
 spin flip rates $c_i (\sigma) $ as a convex combination of local
 range spin flip rates.

To present this decomposition we need the following notation.  
We define
\[
M_i= 2e^{\beta \sum_{B : i \in B } | J_B| },
\]
and the sequence $ (\lambda_i(k))_{k \ge 0}$ as follows
\begin{eqnarray*}
 \lambda_i (k) = \left\{
 \begin{array}{ll}
  { e^{ - 2 \beta \sum_{B : i \in B} | J_B| }} & \mbox{ if } k = 0,\\ 
  e^{ -
 \beta \sum_{B : i \in B, B \not \subset B_i (1) } | J_B|
  } - e^{ - 2 \beta \sum_{B : i \in B} | J_B|  }& \mbox{ if } k = 1,\\ 
  e^{ - \beta \sum_{B : i \in B, B \not \subset B_i (k)} | J_B| } - e^{ - \beta  \sum_{B : i \in B, B \not \subset B_i (k- 1 )} | J_B|  } &
 \mbox{ if } k > 1.
 \end{array} \right.  
 \end{eqnarray*}
 
 Finally, for all $\sigma \in S $ and
for any $k \geq 2,$ we define the update probabilities
\begin{multline*}
 p^{[k]}_i (- \sigma (i) | \sigma )  =\frac{1}{M_i} e^{- \beta
\sum_{B : i \in B, B \subset B_i ( k-1) }  J_B \chi_B ( \sigma)} \\ 
\frac{ e^{- \beta \sum_{B : i \in B, B \subset B_i (k), B \not \subset B_i (k-1)} J_B \chi_B (\sigma)  } - e^{- \beta \sum_{B : i \in B, B \subset B_i (k), 
B \not \subset B_i (k-1)  }| J_B|  }}{ 
1 - e^{- \beta \sum_{B : i \in B, B \subset B_i (k), 
B \not \subset B_i (k-1)  } | J_B|  }} \, ,
\end{multline*}
and for $ k = 1, $
\begin{eqnarray*}
 p^{[1]}_i (-\sigma (i) | \sigma ) = \frac{1}{M_i} \frac{ 
 e^{- \beta
     \sum_{B : i \in B , B \subset B_i (1)   }J_B \chi_B( \sigma)   } - e^{-
     \beta \sum_{ B : i \in B, B \subset B_i (1)  } | J_B| }}
     { 1 - e^{ - 2 \beta
     \sum_{ B : i \in B , B \subset B_i (1)  } | J_B|  } e^{ - \beta \sum_{ B : i \in B, B \not \subset B_i (1) } | J_B| } 
     }.
\end{eqnarray*}
To obtain a probability measure on $A,$ we define for all $k \geq 1,$
$$ p^{[k]}_i (\sigma (i) | \sigma  ) = 1 -p^{[k]}_i (-\sigma (i) | \sigma  ) .$$
Finally, for $k = 0,$ we define 
\begin{equation}\label{p0}  p_i^{[0]} (1) = p_i^{[0]} (-1) = \frac{1}{2} . \end{equation}
Since this last probability does not depend on the site $i$ we omit
the subscript $i$.  Note that, by construction, for any integer $k \ge
1$, the probabilities $ p_i^{[k]} (a|\sigma ) $ depend only on $\sigma
(B_i (k)) .$

Now we   state the decomposition theorem.

\begin{theo}\label{theo:dec} Let us assume that the uniform summability condition 
  (\ref{B3}) holds.
\begin{enumerate}

\item The sequence $ (\lambda_i(k))_{k \ge 0}$ defines a probability
  distribution on the set of positive integers $\{0, 1,...\}$.

\item For any $\ell \ge1$, $\sigma \in S$, the following decomposition
  holds
\begin{equation}\label{eq:decompositionL}
c_i^{[\ell]} ( \sigma) = M_i \left[ \lambda_i (0) \frac12 + \sum_{k=
1}^{\ell} \lambda_i (k) p_i^{[k]} (-\sigma (i)| \sigma )\right] .
\end{equation}

\item For any $\sigma \in S$, the following decomposition holds
\begin{equation}
c_i ( \sigma) = M_i \left[ \lambda_i (0) \frac12 + \sum_{k=
1}^{\infty} \lambda_i (k) p_i^{[k]} (-\sigma (i)| \sigma )\right].
\end{equation}\end{enumerate}
\end{theo}

\begin{proof}
Item 1 follows directly from the definition and \eqref {B3}. 

To prove item 2, set
$$ c^{[0]}_i (\sigma )= \frac12 M_i \lambda_i (0)$$ and observe that
$$ c^{[\ell]}_i (\sigma ) = \sum_{ k = 1}^{\ell} \left[ c_i^{[k]} (
\sigma ) - c_i^{[k-1]} (\sigma ) \right] + c^{[0]}_i (\sigma ) ,$$
where the differences $  c_i^{[k]} (
\sigma ) - c_i^{[k-1]} (\sigma )$ are always positive.
Since by definition
\[ 
 c_i^{[k]} ( \sigma ) - c_i^{[k-1]}( \sigma )= M_i \lambda_i (k ) \,
 p_i^{[k]} (- \sigma (i) |\sigma ) \, ,
\]
we obtain the decomposition stated in item 2.

To prove item 3, observe that the uniform summability condition (\ref{B3}) implies that
\begin{equation}
 \lim_{k \to \infty} c^{[k]}_i (\sigma) = c_i(\sigma )\,  ,
\end{equation}
since $ r_i^{[k]} \to 0 $ for $k \to \infty .$ 
Thus, taking into account item 2, we obtain the desired decomposition.

\end{proof}

To present the coupling, it is convenient to rewrite the generator of
the Glauber dynamics given in \eqref {eq:generator} as
\begin{equation}
  \label{eq:generatormodi}
  \LL \,f(\sigma) \,=\, \sum_{i \in \ZZ^d} \sum_{a \in A} c_i
  (a|\sigma) [f(\sigma^{i, a}) - f(\sigma)] \, ,
\end{equation}
where 
$$\sigma^{i,a}(j) = \sigma(j) \mbox{, for all $j \neq i,$ and
  $\sigma^{i,a}(i) = a,$}$$ and where
\begin{equation}  \label{A21}  c_i (a| \sigma ) =\left\{
  \begin{array}{ll}
 c_i (\sigma) &\mbox { if } a = - \sigma (i), \\ M_i - c_i (\sigma) &
 \mbox { if } a = \sigma (i)\, .
\end{array} 
\right.     \end{equation}

In a similar way, for any integer $\ell \geq 1, $ we define the rates
\begin{equation}  \label{A21B}  c^{[\ell]}_i (a| \sigma ) =\left\{
  \begin{array}{ll}
 c^{[\ell]}_i (\sigma) &\mbox { if } a = - \sigma (i), \\ M_i \sum_{k
 = 0}^\ell \lambda_i (k) - c^{[\ell]} _i (\sigma) & \mbox { if } a =
 \sigma (i)
\end{array} 
\right.     \end{equation}
and the generator 
\begin{equation}
  \label{eq:generatormodiL}
  \LL^{[\ell]} \,f(\sigma) \,=\, \sum_{i \in \ZZ^d} \sum_{a \in A}
  c^{[\ell]}_i (a|\sigma) [f(\sigma^{i,a }) - f(\sigma)] \, .
\end{equation}

This amounts to include in the generators the rates of invisible jumps
in which the spin $i$ is updated with the same value it had before.
Obviously the generator defined in \eqref {eq:generator} (and in
\eqref{A8}) defines the same stochastic dynamics as the generator
given in \eqref {eq:generatormodi} (and in \eqref{eq:generatormodiL},
respectively).

With this representation, we have the following corollary of Theorem 
\ref{theo:dec}. 

\begin{cor}\label{eq:decomposition}
$$ \LL \,f(\sigma) \,=\, \sum_{i \in \ZZ^d} \sum_{a \in A} \sum_{k
 \geq 0} M_i \, \lambda_i(k)p_i^{[k]} (a|\sigma) [f(\sigma^{i,a}) -
 f(\sigma)] \, ,
$$
and for any $\ell \geq 1 , $
$$ \LL^{[\ell]} \,f(\sigma) \,=\, \sum_{i \in \ZZ^d} \sum_{a \in A}
\sum_{k = 0}^{\ell} M_i \, \lambda_i(k)p_i^{[k]} (a|\sigma)
[f(\sigma^{i, a}) - f(\sigma)] \, .
$$
\end{cor}

From now on, we fix $\ell=L$.  The decomposition given in Corollary
\ref{eq:decomposition} suggests the following construction of the
Glauber dynamics having $\LL $ and $\LL^{[L]}$ as infinitesimal
generator.  For each site $i \in \Z^d$ consider a Poisson point
process $N^i$ having rate $M_i$. The Poisson processes corresponding
to distinct sites are all independent. If, at time $t$, the Poisson
clock associated to site $i$ rings, we choose a range $k$ with
probability $\lambda_i (k)$ independently of everything else. Then we
update the value of the configuration at this site by choosing a
symbol $a$ with probability $p_i^{[k]} (a | \sigma (B_i (k)) )$
depending only on the configurations inside the set $B_i (k) .$ It is
clear that using this decomposition, we can construct both processes
$\sigma_t$ and $\sigma_t^{[L]} $ in a coupled way, starting from any
initial configuration.

Actually we can do better than this. We can make a perfect simulation
of the pair of measures $\mu $ and $\mu^{[L]} $ which are the
invariant probability measures of the processes having generators $\LL
$ and $\LL^{[L]}$ respectively.

Let us explain how we sample the configuration at a fixed site $i \in
\Z^d$ under the measure $\mu.$ The simulation
procedure has two stages. In the first stage we determine the set of
sites whose spins influence the spin at site $i$ under equilibrium. We
call this stage backward sketch procedure. It is done by climbing up
from time $0$ back to the past a reverse time Poisson point process
with rate $M_i$ until the last time the Poisson clock rung. At that
time, choose a range $k$ with probability $\lambda_i (k) .$ If $k =
0,$ we decide the value of the spin with probability $\frac12 $
independently of anything else. If $k$ is different from zero, we
restart the above procedure from any of the sites $j \in B_i (k) .$
The procedure stops whenever each site involved in this backward time
evolution has chosen a range $0.$

When this occurs, we can start the second stage, in which we go back
to the future assigning spins to all sites visited during the first
stage. We call this procedure forward spin assignment procedure. This
is done from the past to the future by using the update probabilities
$p_i^ {[k]} $ starting at the sites which ended the first
procedure by choosing range $0.$ For each one of these
sites a spin is chosen by tossing a fair coin.  The values obtained in
this way enter successively in the choice of the values of the spins
depending on a neighborhood of range greater or equal to 1.

In the case of the measure $\mu^{[L]} $ we use the same procedure, but
only considering the choices of $k $ which are smaller or equal to
$L.$ These procedures will be described formally in the next section. 

\section{Perfect simulation of the measure  $\mu $ and coupling }\label{algo}

\subsection { Backward sketch procedure} 

 Before describing formally the two algorithms described at
the end of the last section, let us define the stochastic process
which is behind the backward sketch procedure. Our aim is to define,
for every site $i \in \Z^d ,$ a process $(C_s^{(i)})_{ s \geq 0 }$ taking
values in the set of finite subsets of $\Z^d,$ such that $C_s^{(i)}$
 is the set of
sites at time $- s$ whose spins affect the spin of site $i$ at time $t =
0.$ This is done as follows.

For each $ i \in \Z^d, $ denote by $\ldots T_{-2}^i <T_{-1}^i < T_{0}^i < 0 < T_1^i < T_2^i <
\ldots$ the occurrence times of the rate $M_i$ Poisson point process
$N^i $ on the real line. The Poisson point processes associated to
different sites are independent. To each point $T_n^i$ associate an
independent mark $K^i_n$ according to the probability distribution
$(\lambda_i(k))_{k \ge 0}$. As usual, we identify the Poisson point
processes and the counting measures through the formula
$$N^i[s,t] \,=\, \sum_{n \in \Z} \one_{\{ s \le T_n^i \le t\}}.$$

The backward point process starting at
time $0,$ associated to site $i \in \Z^d $ is defined as
\begin{eqnarray}
  \label{eq:tildet}
  {T}^{(i,0)}_n &=& - T^i_{-n+1}, \mbox{ for any } n \geq 1 . 
 \end{eqnarray}
 We also define the associated marks
\begin{eqnarray}
  \label{eq:tildet1}
  {K}^{(i,0)}_n &=& K^i_{-n+1} 
 \end{eqnarray}
and for any $k \ge 0$, the backward $k$-marked
 Poisson point process starting at time $0$ as 
 \begin{equation}
   \label{eq:tilden}
   {N}^{(i,k)}[s,u] \,=\, \sum_{n \in \Z } \one_{\{s \le
   {T}^{(i,0)}_{n} \le u\}} \one_{\{{K}^{(i, 0 )}_n = k\}}.
 \end{equation}

 To define the backward sketch process we need to introduce a family
 of transformations $\{\pi^{(i,k)}, i \in \Z^d, k \ge 0\}$ on 
 $ {\cal F}(\Z^d),$ the set
 of finite subsets of $\Z^d,$ defined as
 follows. For any unitary set $\{j\} $ and $k \geq 1,$
 \begin{equation}
   \label{eq:pij}
   \pi^{(i,k)}(\{j\}) \,=\, \left\{ \begin{array}{ll} B_i (k), &
                                   \mbox{ if } j=i, \\ \{j\}, & \mbox{
                                   otherwise}.
                                    \end{array} \right.
 \end{equation}
 For $k=0,$ we define
\begin{equation}
   \label{eq:pij0}
   \pi^{(i,0)}(\{j\}) \,=\, \left\{ \begin{array}{ll}
                                   \emptyset , & \mbox{ if } j=i,\\
                                   \{j\}, & \mbox{ otherwise}.
                                    \end{array} \right.
 \end{equation}
For any set finite set $F \subset \Z^d$, we define similarly
\begin{equation}
  \label{eq:pif}
  \pi^{(i,k)}(F) \,=\, \cup_{j \in F} \pi^{(i,k)}(\{j\}) . 
\end{equation}
The backward sketch process starting at site $i$ at time $0$ will be
denoted by $(C_s^{(i)})_{s \geq 0}.$ The set $C_s^{(i)}$ is the set of
sites at time $- s$ whose spins affect the spin of site $i$ at time $t =
0.$

The evolution of this process is defined through the following
equation. $C_0^{(i)} = \{i\},$ and
\begin{equation}
  \label{eq:ct}
  f( C_s^{(i)}) \,=\, f(C_0^{(i)}) \,+\, \sum_{k \ge 0} \sum_{j \in
  \Z^d} \int_0^s [f(\pi^{(j,k)} (C_{u-}^{(i)})) - f(C_{u-}^{(i)})]\,
  {N}^{(j, k)}(du),
\end{equation}
where $f: {\cal F}(\Z^d) \rightarrow \R$ is any bounded cylinder
function. This family of equations characterizes completely the time
evolution $(C_s^{(i)})_{ s \ge 0}$. 

In a similar way, for the truncated process,
we define its associated backward sketch process by
\begin{equation}
  \label{eq:ctL}
  f( C_s^{[L], (i)}) \,=\, f(C_0^{[L],(i)}) \,+\, \sum_{k = 0}^L
  \sum_{j \in \Z^d} \int_0^s [f(\pi^{(j,k)} (C_{u-}^{[L],(i)})) -
  f(C_{u-}^{[L],(i)})]\, {N}^{(j, k)}(du),
\end{equation}
where we use the same Poisson point processes ${N}^{(j, k)}, 0 \le k
\le L, $ as in (\ref{eq:ct}).

The following proposition summarizes the properties of the family of
processes defined above.

\begin{prop}
  For any site $i \in \Z^d$, $C_s^{(i)}$ and $C^{[L], (i)}_s$, $s \in
  \R^+$, are Markov jump processes having as infinitesimal generator,
\begin{equation}
  \label{eq:generatord}
  {\cal L} f(C) \,=\, \sum_{i \in C} \sum_{k \ge 1} M_i \lambda_i (k) [f(C
  \cup B_i(k)) - f(C)] + \lambda_i (0) [f(C \setminus \{i\}) - f(C)] ,
  \end{equation}
 with  initial condition at time $t=0$,
$C_0^{(i) } = \{i\}$,
  and \begin{equation}
  \label{eq:generatordt}
 {\cal L}^{[L]} f(C) \,=\, \sum_{i \in C} \sum_{k= 1}^L M_i \lambda_i (k)
 [f(C \cup B_i(k)) - f(C)] + \lambda_i (0) [f(C \setminus \{i\}) -
 f(C)] ,
\end{equation} 
  with initial condition at time $t = 0 $, $C^{[L], (i)}_0 = \{i\}, $
  respectively. Here $f: \F (\Z^d) \to \R $ is any bounded cylindrical
  function.
 \end{prop}
Let 
$$ T^{(i)}_{STOP} = \inf \{ s : C_s^{(i)} = \emptyset \} $$
be the relaxation time. 
We
introduce the sequence of successive jump times $\tilde T_n^{(i)} , n
\geq 1 ,$ of processes $N^{(j,k)} $ whose jumps occur in
(\ref{eq:ct}).  Let $\tilde T_1^{(i)} = T_1^{(i,0)} $ and define
successively for $n\geq 2 $
\begin{equation}
  \label {Ajumps} \tilde T_n^{(i)} = \inf \{ t > \tilde T_{n-1}^{(i)} : \, 
\exists j \in C^{(i)}_{\tilde T_{n-1}^{(i)} } , \exists k :
N^{(j,k)} ( ]\tilde T_{n-1}^{(i)} , t ]) = 1 \} . \end{equation}
Now we put
\begin{equation}
  \label {RR5} {\bf C}^{(i)}_n = C^{(i)}_{ \tilde T_n^{(i)} } .\end{equation} Finally, notice that  $ N^{(i)}_{STOP}$ of Definition  \ref {S.2} is equal to 
$$ N^{(i)}_{STOP} = \inf \{ n : {\bf C}^{(i)}_n = \emptyset \} .$$
This is the number of steps of the backward sketch process. 
For the algorithm to be successful, it is crucial to show that both relaxation time
$ T^{(i)}_{STOP}$ and the number of steps $ N^{(i)}_{STOP} $ are finite. This is  the content of next  theorem.

\begin {prop}\label {nstop}
 Under the requirement  
  \begin{equation}
    \label{eq:condition2}
 \sup_{i \in \Z^d}   \sum_{k \ge 1} \, |B_i (k)|  \lambda_i (k)  \,< \,1,
  \end{equation} we have  uniformly for any $i  \in \ZZ^d , $
\begin{equation}
    \label{s3} P(T^{(i)}_{STOP} > t ) \le  e^{ - \gamma  t} 
    \end{equation}
and 
 \begin{equation}
    \label{s4} E( N^{(i)}_{STOP}) \le \frac1\gamma,   \end{equation}
where 
\begin{equation}\label{eq:epsilon}
  \gamma= 1 - \sup_{i \in \Z^d} \sum_{k \ge 1} \, |B_i (k)|  \lambda_i (k).
\end{equation}
   \end{prop}
The proof of  Proposition \ref{nstop} will be given at the end of Section \ref{sectionafter} below.

\subsection {Algorithms} 

We now go to the crucial point of precisely defining the maps $F$ and $G$
whose existence is claimed in Theorem \ref{theo:perfect}. 
These  maps are  implicitly defined by the backward sketch procedure
and the forward spin assignment procedure. Let us give the
algorithmic like description of these procedures.    
The following variables will be used. 

\begin{itemize}
\item $N$ is an auxiliary variables taking values in the set of
non-negative integers $ \{ 0, 1,2, \ldots \} $
\item $N^{(i)}_{STOP}$ is a counter taking values in the set of
non-negative integers $ \{ 0, 1, 2, \ldots \} $
\item  $N^{[L], (i)}_{STOP}$ is a counter taking values in the set of non-negative integers
$ \{ 0, 1, 2, \ldots \} $
\item
$I $ is a variable taking values in $\Z^d$ 
\item
$I^{[L]} $ is a variable taking values in $\Z^d$
\item
$K$ is a variable taking values in $\{  0, 1, \ldots \}$
\item
$K^{[L]}$ is a variable taking values in $\{  0, 1, \ldots \}$
\item
$B $ is an array of
elements of $\Z^d \times \{ 0 , 1 , \ldots \} $
\item
$B^{[L]}$ is an array of
elements of $\Z^d \times \{  0 , 1 , \ldots \} $
\item
$C$ is a variable taking values in the set of finite subsets of $\Z^d$ 
\item
$C^{[L]}$ is a variable taking values in the set of finite subsets of $\Z^d$ 
\item $ W  $ is an auxiliary variable taking values in $ A$
\item $X $ is a function from $\Z^d $ to $A \cup \{ \Delta \} ,$
  where $\Delta $ is some extra symbol that does not belong to $A$
\item $X^{[L]} $ is a function from $\Z^d $ to $A \cup \{ \Delta \} ,$
  where $\Delta $ is some extra symbol that does not belong to $A$

\end{itemize}
We will present  the  backward sketch procedure only for  $G$, the coupled perfect simulation 
algorithm. The backward sketch procedure for  $F$ can be immediately  deduced ignoring the steps  from $11.$ to $18.$  concerning the variables with superscript $L$.

\underline{{\bf Algorithm 1 } Backward sketch procedure}
 
\begin{enumerate}
\item
{\it Input:} $\{i\} $; {\it Output:} $N^{(i)}_{STOP}$,  $N^{[L], (i)}_{STOP}$, $B,$ $B^{[L]}$ 
\item
 $N \leftarrow 0,$ $N^{(i)}_{STOP} \leftarrow 0 ,$ $N^{[L], (i)}_{STOP} \leftarrow 0 ,$ $ B \leftarrow \emptyset ,$ $ B^{[L]} \leftarrow \emptyset ,$
$ C \leftarrow \{i\} , $
$ C^{[L]} \leftarrow \{i\}  $  
\item
WHILE {$C \neq \emptyset $ } 
\item
$N \leftarrow N+1 $ 
\item Choose independent random variables $S_N^{(j,k)} $ having exponential distribution of parameter $M_j \lambda_j(k) $ for any $j \in C$ and for any $ k \geq 0$ 
\item Choose $(I, K) = \arg \min \{ S_N^{(j,k)} , j \in C , k \geq 0 \} $
\item IF {$K = 0,$} {$C \leftarrow C \setminus \{ I \}$}
\item
ELSE  $ C \leftarrow C \cup B_I (K)$
\item
ENDIF
\item $ B (N) \leftarrow  (I, K ) $
\item
IF {$C^{[L]} \neq \emptyset $} 
\item \hspace{1cm}  $N^{[L], (i)}_{STOP} \leftarrow N $
\item\hspace{1cm} Choose $ (I^{[L]}, K^{[L]}) = \arg \min \{ S_N^{(j,k)} , j \in C , 0 \le k \le L \} $
\item \hspace{1cm}
IF {$K^{[L]} = 0,$} 
 {$C^{[L]} \leftarrow C^{[L]} \setminus \{ I^{[L]} \}$}
\item \hspace{1cm}
ELSE 
 $ C^{[L]} \leftarrow C^{[L]} \cup B_{I^{[L]}} (K^{[L]})$
\item \hspace{1cm}
ENDIF
\item \hspace{1cm} $ B^{[L]} (N) \leftarrow  (I^{[L]}, K^{[L]} )$
\item
ENDIF
\item
ENDWHILE
\item $N^{(i)}_{STOP} \leftarrow N $
\item
RETURN  $N^{(i)}_{STOP}$, $ N^{[L], (i)}_{STOP},$ $B$, $B^{[L]}$  
\end{enumerate}

If $B = B^{[L]},$ then we use the following Forward spin assignment
procedure to sample $X (i) = X^{[L]} (i) .$ This is the algorithmic
translation of the ideas presented in the last paragraph of Section 4.
If $B \neq B^{[L]},$ then we use the Forward spin assignment
procedure twice independently, starting with input $N^{(i)}_{STOP}$,
$B$ in order to sample $X(i) $ and starting with input $N^{[L],
(i)}_{STOP}$, $B^{[L]}$ in order to sample $X^{[L]} (i) .$ \\

\underline{{\bf Algorithm 2} Forward spin assignment procedure}
\begin{enumerate}
\item {\it Input:} $N^{(i)}_{STOP}$, $B$; {\it Output:} $\{X(i) \}$ 
\item $ N \leftarrow N^{(i)}_{STOP }   $
\item $X (j) \leftarrow \Delta $ for all $ j \in  \Z^d  $
\item WHILE {$N \ge 1$}
\item $ (I,K) \leftarrow B(N)  .$ 
\item IF {$K= 0 $} choose $W $ randomly in $A$
according to the probability distribution
$$ P( W = v)  = p_I^{[0]} (v ) = \frac 12 $$
\item ELSE {choose $W $ randomly in $A$
according to the probability distribution
$$ P( W  = v)  = p_I^{[K]} ( v | X )$$}
\item ENDIF
\item $ X (I) \leftarrow W $
\item { $ N \leftarrow N-1 $}
\item ENDWHILE 
\item RETURN $ \{ X (i)  \}  $
\end{enumerate}

\begin{rem} At a first look to steps 5 and 6 of Algorithm 1, the 
reader might think that the simulation of an infinite number of
exponential variables is necessary in order to perform the
algorithm. Actually, it is sufficient to simulate a finite number of
finite valued random variables, see Knuth and Yao (1976) and also
Section 10 of Galves et al. (2008).
\end{rem}

\section{Proof of Theorem  \ref{theo:perfect}   }
\label{sectionafter}
This section is devoted to the proof of Theorem  \ref{theo:perfect}. 
For the convenience of the reader, we start by recalling the following theorem of Galves et al. (2008). 

\begin {theo}\label {theo:GGL1}
 Under the  requirement  
 \eqref{eq:condition2}, 
 $$\sup_{i \in \Z^d}   \sum_{k \ge 1} \, |B_i (k)|  \lambda_i (k)  \,< \,1,
$$
the Gibbs measure $\mu $ is the unique invariant probability measure
of the process $(\sigma_t)_t.$
The output $X(i) $ obtained using successively Algorithms 1 and 2
given in Section \ref{algo} is a perfect sampling of the Gibbs measure
$\mu .$ Finally, the estimate (\ref{s3}) holds :
$$P(T^{(i)}_{STOP} > t ) \le  e^{ - \gamma  t}  .$$
\end{theo}

The same result holds true also for $X^{[L]} $ obtained using
successively Algorithms 1 and 2 but now with restriction that only
ranges $k \le L $ are considered. 

\begin{rem}
In Galves et al. (2008) the above result  was proved for the minimal decomposition
of the spin flip rates. The same proof works for the decomposition of the spin flip rates $c_i(\sigma)$ 
considered in Section 4. Note that in contrary to the decomposition considered in Galves et al. (2008),
the decomposition in Section 4 is explicit in  terms of the interaction. 
\end{rem}

We first check that the
assumption (\ref{cn1}) of Theorem \ref{theo:perfect} implies condition
(\ref{eq:condition2}) of Theorem \ref{nstop}. 

\begin{lem}\label{lemma:1}
Under the conditions of Theorem \ref{theo:perfect}, condition
(\ref{eq:condition2}) is satisfied.
\end{lem}

\begin{proof} 
For any $k \geq 0,$ write for short 
$$ S_i^{ > k} = \sum_{ B : i \in B , B \not \subset B_i (k)}| J_B| .$$
For each $i \in\Z^d$ we have that
\begin{equation}     \label {supcrit2}  \begin {split}     \sum_{k=1}^ \infty  { \lambda}_i (k ) |B_i(k)| & \le    2d  {\lambda}_i (1)  +   \sum_{k=2}^\infty  | B_i (k)|  \left [  e^{ - \beta S_i^{ > k} } -  e^{ - \beta S_i^{ > k-1}} \right ]  
\cr & = 2d  e^{ - \beta S_i^{ > 1 }} \left (1-  e^{ - \beta S_i^{ > 0} }  e^{ - \beta \sum_{B : i \in B, B \subset B_i (1) } | J_B| }\right )  \cr &\quad  + 
\sum_{k=2}^\infty   | B_i (k)|   e^{ - \beta S_i^{ > k} } \left [1  -  e^{ - \beta \sum_{B : i \in B, B \subset B_i (k), B \not \subset B_i (k-1) } | J_B| }\right ] \cr &   \le 
 2d  e^{ - \beta S_i^{ > 1 }} \left (1-  e^{ - \beta S_i^{ > 0} }  e^{ - \beta \sum_{B : i \in B, B \subset B_i (1) } | J_B| }\right )  \cr &\quad  + 
    \beta \sup_{ i } \left( \sum_{k=2}^\infty | B_i (k)|      \sum_{B : i \in B, B \subset B_i (k), B \not \subset B_i (k-1) } | J_B|  \right)  .   \end {split}
 \end{equation}
Under the conditions of Theorem \ref{theo:perfect}, this last expression is finite. Therefore condition  \eqref {eq:condition2} reads as  
\begin{multline*}
2d  \sup_i \left( e^{ - \beta S_i^{ > 1 }} \left (1-  e^{ - \beta S_i^{ > 0} }  e^{ - \beta \sum_{B : i \in B, B \subset B_i (1) } | J_B|}\right )\right)   \\
+ 
    \beta \sup_{ i } \left( \sum_{k=2}^\infty | B_i (k)|      \sum_{B : i \in B, B \subset B_i (k), B \not \subset B_i (k-1) } | J_B| \right)   <1.
    \end{multline*}
Finally, let $ \beta_c$ be the solution of  
 \begin{multline}     \label {betac1}   2d  \sup_i \left( e^{ - \beta S_i^{ > 1 }} \left (1-  e^{ - \beta S_i^{ > 0} }  e^{ - \beta \sum_{B : i \in B, B \subset B_i (1) } | J_B|} \right )\right)   \\
+ 
    \beta \sup_{ i } \left( \sum_{k=2}^\infty | B_i (k)|      \sum_{B : i \in B, B \subset B_i (k), B \not \subset B_i (k-1) } | J_B|  \right)  =1.
    \end{multline}
This concludes the proof. 
\end{proof}

\vskip0.5cm \noindent
{\bf Proof of Theorem 
\ref{theo:perfect}}. 

  {\bf  Proof of  item 1}.  Theorem
\ref{theo:GGL1} combined with Lemma \ref{lemma:1}  proves that the  output    of the perfect sampling algorithm is indeed  a sampling of $\mu$.   Next  we prove  \eqref {s1}.
Define
$$ L^{(i)}_n = | {\bf C}^{(i)}_n|,$$ the cardinal of the set ${\bf
C}^{(i)}_n$ after $n$ steps of the algorithm. Let $(D^{(i)}_n)_{n \geq
0, i \in \Z^d} $ be i.i.d. random variables, independent of the
process, taking values in $\{ -1 , 0 , 1, 2, \ldots \} $ such that for $k \ge 1$ 
\begin {equation} \label {alg1} P( D^{(i)}_n = |B_i(k) | - 1 ) = \lambda_i (k) , \end {equation}
and for $k=0$
\begin {equation} \label {alg1} P( D^{(i)}_n = - 1 ) = \lambda_i (0).  \end {equation}
Note that under assumption (\ref{cn1}), the condition
(\ref{eq:condition2}) holds (see Lemma 1). Hence, $\sup_{i \in \Z^d }
E(D^{(i)} _1) \le -\gamma < 0.$ For any $i \in \ZZ^d$,  let us call
$I^{(i)}_n $ the index of the site whose Poisson clock rings at the
$n-$th jump of the process $C^{(i)} .$ Recall that $I_n^{(i)}$ are
conditionally independent given the sequence $({\bf C}^{(i)}_{n})_n$
such that
$$ P( I_n^{(i)} = k | {\bf C}^{(i)}_{n-1 }) = \frac{M_k}{\sum_{j \in
{\bf C}^{(i)}_{n-1}} M_j }.$$ Put
$$S^{(i)}_n = \sum_{k= 1}^n D_k^{I^{(i)}_k}.$$ Note that by
construction, $S^{(i)}_n + n \gamma $ is a super-martingale.  Then a
very rough upper bound is
\begin{equation}\label{s7}
 L^{(i)}_n \le 1 + S^{(i)}_n \mbox{ as long as } n \le
V^{(i)}_{STOP} ,
\end{equation}
where $V^{(i)}_{STOP} $ is defined as
$$ V^{(i)}_{STOP} = \min \{ k : S^{(i)}_k = -1 \} .$$
By construction
$$ N^{(i)}_{STOP} \le V^{(i)}_{STOP}.  $$ Fix a truncation level $N.$
Then by the stopping rule for super-martingales, we have that
$$ E( S^{(i)}_{V^{(i)}_{STOP} \wedge N } ) + \gamma E( V^{(i)}_{STOP}
\wedge N ) \le 0 .$$ But notice that
$$ E( S^{(i)}_{V^{(i)}_{STOP} \wedge N }) = - 1 \cdot P(
V^{(i)}_{STOP} \le N ) + E (S^{(i)}_N ; V^{(i)}_{STOP } > N ) .$$ On $
\{ V^{(i)}_{STOP } > N \}, $ $S^{(i)}_N \geq 0,$ hence we have that $
E( S^{(i)}_{V^{(i)}_{STOP} \wedge N }) \geq - P( V^{(i)}_{STOP} \le N)
.$ We conclude that
\begin{eqnarray*}
E( V^{(i)}_{STOP} \wedge N ) & \le & \frac{1}{\gamma }  P( V^{(i)}_{STOP} \le N ) .
\end{eqnarray*}
Now, letting $N \to \infty ,$ we get 
$$ E( V^{(i)}_{STOP}) \le \frac1\gamma ,$$ 
and therefore
$$ E( N^{(i)}_{STOP}) \le \frac1\gamma. $$ 

 {\bf Proof of  item 2}.
 We have to show that the coupled perfect sampling algorithm 
$G$ defined in Section 4 achieves the bound (\ref{cn2}). For that sake we introduce the following stopping times.  For
any site $j \in \Z^d $ and any $t \geq 0,$ let
$$ T_1^{(j,t)} = \inf \{ T_n^{(j,0)} > t \} - t $$ be the first jump
of the Poisson point process $ (T_n^{(j,0)})_n$ shifted by time $t.$

Let us call $T^{(i)}_L$ the first time that a range of order $k > L $
has been chosen. $T^{(i)}_L$ is given by
\begin{equation}\label{eq:taul}
T^{(i)}_L = \inf \{ t > 0 : \sum_{ j \in C_t^{(i)} } \sum_{k > L }
N^{(j,k)} ([ t, t + T_1^{(j,t)} [) \geq 1 \} .
\end{equation}

Recall that in order to construct $\mu $ and $\mu^{[L]} ,$ we use the
same Poisson point processes for all ranges $k \le L.$ Thus we have
that
\begin{equation}\label{eq:tau2} P( X(i) \neq X^{[L]}(i)) \le P( T^{(i)}_L \le T^{(i)}_{STOP}). \end{equation}
By Lemma \ref{lemma2} given below 
we  conclude  the proof of Theorem \ref{theo:perfect}.
{\hfill $\bullet$ \vspace{0.25cm}}

\begin{lem}\label{lemma2}
$$ P( T^{(i)}_L \le T^{(i)}_{STOP}) \le \sup_{ i \in \Z^d } \left( 1 -
e^{- \beta \sum_{B : i \in B, B \not \subset  B_i ( L) } | J_B| } \right) \frac 1 \gamma .$$
\end{lem}

\begin{proof}
Put 
$$ \alpha_i (k) = \sum_{ \ell \le k } \lambda_i(\ell) .$$ Given the
structure of the interaction, we have for $k \geq 2,$ $\alpha_i (k) =
e^{ - \beta \sum_{ B : i \in B, B \not \subset B_i ( k) } | J_B|}$ and thus $ 1 -
\alpha_i (L ) \le \sup_{ i \in \Z^d } \left( 1 - e^{- \beta \sum_{ B : i \in B, B \not \subset B_i ( L) } | J_B|}\right).$ Call this last quantity
$$ \delta (L) = \sup_{ i \in \Z^d } \left( 1 - e^{- \beta \sum_{ B : i \in B, B \not \subset B_i ( L) } | J_B|}\right) .$$
Then we have
\begin{eqnarray*}
&&P( T^{(i)}_L \le T^{(i)}_{STOP} )\\ && \le \sum_{n \geq 1 } E\left(
P \left(\left\{ N^{(j,k)}( ] \tilde T_{n-1}^{(i)} , \tilde T_n^{(i)} ]
) = 1 \mbox{ for } j \in {\bf C}^{(i)}_{n-1}, k > L \right\} \, \Big|
\; {\bf C}^{(i)}_{n-1}, N^{(i)}_{STOP} > n-1 \right) \right) \\ && =
\sum_{n \geq 1 } E\left( \sum_{j \in {\bf C}^{(i)}_{n-1}} \frac{M_j
(1- \alpha_j (L)) }{ \sum_{k \in {\bf C}^{(i)}_{n-1}} M_k } 1_{
\{N^{(i)}_{STOP} > n-1\}} \right) \\ && \le \sum_{n \geq 1 } E\left(
\sum_{j \in {\bf C}^{(i)}_{n-1}} \frac{M_j \delta (L) }{ \sum_{k \in
{\bf C}^{(i)}_{n-1}} M_k } 1_{\{ N^{(i)}_{STOP} > n-1 \}} \right) \\
&& = \delta (L) E( N^{(i)}_{STOP}) .
\end{eqnarray*}
In the above calculus we used that given $ {\bf C}^{(i)}_{n-1}$,
$\tilde T_n^{(i)} $ defined in \eqref {Ajumps} is a jump of $N^{(j,k)}
$ with probability
$$ \frac{ M_j \lambda_j (k) }{ \sum_{k \in {\bf C}^{(i)}_{n-1}} M_k
}.$$ Summing over all possibilities $ k > L $ gives the term
$( M_j (1- \alpha_j (L)) )/( \sum_{k \in {\bf C}^{(i)}_{n-1}} M_k ) 
.$
By    (\ref{s4}) we get the result. 
\end{proof}

{\bf Proof of  Proposition  \ref{nstop}}
(\ref{s3}) is the statement of Theorem \ref{theo:GGL1}. (\ref{s4}) has been proved above in the proof of (\ref{s1}). 
{\hfill $\bullet$ \vspace{0.25cm}}

\section{Final discussion}\label{Dob1}
It is worth to note that Theorem \ref{theo:perfect} immediately
provides an upper bound for the $\bar d-$distance of $\mu $ and
$\mu^{[L]} .$ The $\bar d-$distance is defined as follows.

\begin{defin}
  Given two probability measures $\mu $ and $\nu $ on $S,$ a coupling
  between $\mu $ and $\nu $ is a probability measure on $S \times S$
  having as first and second marginals $\mu $ and $\nu$, respectively.
  The set of all couplings between $\mu $ and $\nu$ will be denoted
  ${\cal M} ( \mu , \nu )$.
\end{defin}

In the next definition, the elements of the product space $S \times S$
will be denoted by $ ((\sigma_1 (i), \sigma_2 (i)),  i\in \ZZ^d )  .$

\begin{defin}
  The distance $\bar d $ between two probability measures $\nu_1 $ and
  $\nu_2 $ on $S$ is defined as
$$ \bar d ( \nu_1 , \nu_2 ) = \inf_{ Q \in {\cal M} ( \nu_1 , \nu_2 )}
\left \{ \sup_{i \in \ZZ^d } Q ( \sigma_1 (i) \neq \sigma_2 (i) )
\right\} .
$$ 
\end{defin}
This definition naturally extends Ornstein's $\bar d -$distance to the
space of non-homogeneous random fields. As a corollary of Theorem
\ref{theo:perfect} we obtain

 \begin {theo} \label {theo:coupling}
 Under the assumptions of Theorem \ref {theo:perfect} we have
\begin{equation}  \label {bound1}\begin {split} &
    \bar d ( \mu , \mu^{[L]} ) \le C \sup_{i \in \Z^d } \left( 1-
    e^{-\beta \sum_{ B : i \in B, B \not \subset B_i ( L) } | J_B|} \right) ,
\end {split}
\end{equation}
where $C$ is  the same constant as in \eqref  {cn2}.  
\end {theo} 

Upper bounds on the $\bar d-$distance can as well be obtained by using
a contraction argument. This approach was introduced in Dobrushin (1968-1970)
to prove uniqueness of the infinite volume Gibbs
measure. Applied to our context, it implies the existence of a good
coupling through compactness arguments.  

For a nice and self contained presentation of this method, we refer
the reader to Presutti (2009), see also Chapter 8 of Georgii.
Applying Dobrushin's method, we obtain the following upper bound on
Ornstein's $\bar d-$distance.

\begin {theo} \label {theo:dob}
   Assume that
   \begin{equation}\label {main1} \beta \sup_i  \sum_{B : i \in B } |J_B|  = r <1.
   \end{equation}
Then there exist unique infinite volume Gibbs measures $\mu$ and
$\mu^{[L]}$ and they satisfy
\begin{equation}  \label {bound2}
    \bar d ( \mu , \mu^{[L]} ) \le \frac { \beta } {1-r} \left (
    \sup_{i \in \Z^d } \sum_{ B : i \in B, B \not \subset B_i ( L) } |J_B| \right
    ).
 \end{equation}
\end {theo} 

\begin{proof}
The proof of Theorem \ref{theo:dob} can be achieved following the
ideas in Dobrushin (1968-1970), exposed in Chapter 3 of Presutti (2009).
\end{proof}

Observe
that condition \eqref {main1} is weaker than the one requested in
Theorem \ref {theo:perfect}.  In particular, this condition is too
weak to guarantee that our perfect simulation procedure of $\mu $ and
$\mu^{[L]} $ stops after a finite number of steps.  As a counterpart
of the less restrictive assumption \eqref {main1}, the upper bound
\eqref {bound2} is not as good as the upper bound \eqref {bound1}.

However, from our point of view, the important fact is that the contraction method of
Dobrushin is not constructive and does not provide an explicit
sampling procedure for both measures. To provide an explicit coupled perfect
simulation procedure for the measures is precisely the goal of the present paper. 

Further features of high temperature region can be deduced by our constructive approach.  We  discuss  as an example the decay of correlations between localized observables as the regions of localization are moved away from each other.

This decay of correlations can be described in terms of an associated random walk that 
we introduce now. Suppose that $\mu $ is translational invariant. In this case, the sequence $(\lambda_i (k))_k$ does not depend on $i,$ and we shall write $\lambda (k) = \lambda_i (k),$ for any $ k \geq 0 .$  
Let $ \xi_1, \xi_2, \ldots $ be i.i.d. random variables with common distribution
$$ P ( \xi_1 =- 1) = \lambda  (0),  \quad P ( \xi_1 = |B_0(k)|- 1) = \lambda  (k), \quad k \ge 1.$$
 Put $ \varphi (\lambda ) = E ( e^{\lambda \xi_1 }) $ and let 
$$ \rho = \sup \{ \lambda > 0 : \varphi ( \lambda ) \le 1 \} .$$
Finally, consider the following left-continuous random walk $S_n= \xi_1 + \ldots + \xi_n, $  for  $n\geq 0$. 
Note  that $S_n$  is an upper bound on  the total number of sites   whose spins affect the spin  at site $i$ under equilibrium, after $n$ steps of the algorithm.    Under condition \eqref{eq:condition2}, $E ( \xi_1) = - \gamma < 0,$ hence $S_n$ drifts to $ - \infty .$  
This implies that $ M = \sup_n S_n, $ the maximum of the random walk, is finite almost surely.

The following theorem is an immediate consequence of our construction  and of results of Korshunov (1997) concerning the  maximum of a random walk.
In order to state the theorem, we have to recall the notion of locally power functions.
\begin{defin}
A function $f : \R_+ \to \R_+$ is called locally power if for every fixed $t \geq 0, $ $f(x+t) \sim f(x) $ as $x \to \infty,$ i.e. $\lim_{x \to \infty} \frac {f(x+t)} {f(x)} = 1$. 
\end{defin}

\begin{theo}\label{prop:mixing1}
  Take two cylinder functions $f$ and $g$ having support $ \Delta_f\Subset \Z^d $ and  $ \Delta_g\Subset \Z^d$. 
Assume the distance  $d( \Delta_f,  \Delta_g)= R $.  Under  the  assumption (\ref{eq:condition2})   we have  
\begin{equation}\label{eq:mixing}
 | \mu (fg)- \mu(f) \mu(g)| \le 2 c_f c_g ( | \Delta_f| + |\Delta_g|) \;  P ( M\geq R/2)     ,
 \end{equation} 
where
 \begin{equation}  \label {en1} c_f = 
 \sup_{i \in \Delta} \sup_{\sigma \in S}  | f(\sigma)- f(\sigma^i) |  
\end{equation} 
and $c_g$ is similarly defined. 
Concerning the decay of $P ( M\geq R/2),$ the following holds.
\begin{enumerate}
\item
Suppose that $ \rho = 0 $ and that $ P ( \xi_1 + \xi_2 \geq n ) \sim 2 P (\xi_1 \geq n ) $ as $n \to \infty .$
Then   as $ n \to \infty ,$
$$ P ( M\geq  |B_0 (n)|) \sim \frac{1}{\gamma} \left[ \sum_{ k = n +1}^ \infty \lambda (k) |B_0 (k)| -  | B_0 (n) | \left( \sum_{ k > n+1} \lambda (k) \right) \right]  $$
    where $\gamma $ is given in \eqref{eq:epsilon}.
\item
Suppose that $ \rho > 0,$ $ \varphi (\rho ) < 1 ,$ that $ P ( \xi_1 + \xi_2 \geq n ) \sim c  P (\xi_1 \geq n ) $ as $n \to \infty $ for some constant $c$ and suppose that $e^{ \rho n} P( \xi_1 \geq n ) $ is locally power. Then  as $ n \to \infty $
$$P ( M\geq  |B_0 (n)| ) \sim \frac{\rho E ( e^{\rho M })}{ 1 - \varphi (\rho)}    \sum_{ k = n +1}^ \infty \lambda (k).    $$
 \item
Suppose that $\rho > 0 ,$ $\varphi (\rho) = 1$ and $\varphi' (\rho) < + \infty .$ Then
$$ P ( M\geq  |B_0 (n)|) \sim c e^{ - \rho n } $$
as $ n \to \infty ,$ where $c$ is some positive constant. 
\end{enumerate}
\end{theo}

\begin{rem}
Using the same type of upper bounds as in the proof of Lemma \ref{lemma:1}, note that the expression 
$$ \sum_{ k = n +1}^ \infty \lambda (k) |B_0 (k)|  $$
that defines the decay of correlations can be upper bounded by 
$$ \beta \sum_{ k = n +1}^ \infty    \left( \sum_{ B : 0 \in B , B \subset B_0 ( k) , B \not \subset B_0 ( k - 1) } | J_B| \right) |B_0 (k)|  ,$$
which yields the tail of the converging series of condition (\ref{cn1}). 
\end{rem}

\begin{proof}
We start by proving (\ref{eq:mixing}).
It is sufficient to prove it for functions $f$ and $g$ having support
$\Delta_f = \{i \}$ and $\Delta_g = \{ j \} ,$ $ R = \| i - j \| > 0  .$ 
For any site $k \in \Z^d,$ let 
\begin{equation}
{\bf C}^{(k)} = \bigcup_{n \geq 0 } {\bf C_n^{(k)} } 
\end{equation}
be the total number of sites whose spins affect the spin of site $k$ under equilibrium (recall \eqref{RR5}). Then it is evident that 
\begin{equation}\label{eq:firststep}
  | \mu (fg)  - \mu (f) \mu (g) | \le c_f c_g ( | \Delta_f | + | \Delta_g | )  P ( {\bf C}^{(i)} \cap {\bf C}^{(j)} \neq  \emptyset ) .
\end{equation}
Then the same comparison argument that leads to (\ref{s7}) yields 
$$   P ( {\bf C}^{(i)} \cap {\bf C}^{(j)} \neq  \emptyset ) \le P ( M + M' \geq R ) \le 2 P ( M \geq R/2 ) .$$ 
This concludes the proof of (\ref{eq:mixing}).

Item  1   is a consequence of Theorem  1   of Korshunov (1997), see \cite{Korshunov}. First notice that 
$\bar F(t) = P ( \xi_1 > t) $ is constant on intervals $[ |B_0 (n) | - 1, | B_0 ( n+1) | - 1 [. $  Then evidently $\bar F ( |B_0 (n) | - 1 ) = 1 - \alpha (n) ,$ where $ \alpha _n = \lambda_ 0 + \ldots + \lambda_n .$ 
Therefore, 
\begin{eqnarray*}
\int_{  |B_0 (n) |}^ \infty \bar F (t) dt &=& 
\left ( \sum_{ k = n}^ \infty ( 1 - \alpha (k)) \left[ | B_0 ( k+1)| - |B_0 (k) | \right] \right ) -
 ( 1 - \alpha (n))   \\
&=& \left( \sum_{ k = n+1}^ \infty   \lambda (k) |B_0 (k) | \right)  - ( 1 - \alpha ( n)) [  | B_0 (n) | +1]\\
&=& \left( \sum_{ k = n+1 }^ \infty   \lambda (k) |B_0 (k) | \right)  -  [  | B_0 (n) | +1]  \left( \sum_{ k > n} \lambda (k) \right) . 
\end{eqnarray*}
This integral $\int_{ |B_0 (n) | }^ \infty \bar F (t) dt$ is the quantity that
determines the asymptotic behavior of $ P ( M\geq  |B_0 (n) | ) $ as it is shown in Theorem   1  of Korshunov (1997).
Items 2 and   3 are  a consequence of Theorem 2 and Theorem  3 (ii) of Korshunov (1997). 
\end{proof}

\vskip30pt
\newpage

Antonio Galves

Instituto de Matem\'atica e Estat\'{\i}stica

Universidade de S\~ao Paulo

Caixa Postal 66281

05315-970 S\~ao Paulo, Brasil

e-mail: {\tt galves@usp.br}
\bigskip

Eva L\"ocherbach

LAMA 

Universit\'e Paris 12 Val de Marne

61 avenue du g\'en\'eral de Gaulle

94 000 CRETEIL CEDEX,  France

email: {\tt locherbach@univ-paris12.fr}

\bigskip

  Enza Orlandi
 
Dipartimento di Matematica

Universit\`a  di Roma Tre

 L.go S.Murialdo 1, 00146 Roma,  Italy. 

email: {\tt orlandi@mat.uniroma3.it}
 \end{document}